\documentclass[10pt]{amsart}
\usepackage[cp1251]{inputenc}
\usepackage[english,russian]{babel}
\usepackage{amsmath}
\usepackage{amssymb}
\usepackage{amsfonts}

\newtheorem{lemma}{Lemma}
\newtheorem{theorem}{Theorem}

\newtheorem{proposition}{Proposition}

\usepackage{amscd,euscript,amsxtra,amsthm}

\theoremstyle{definition}
\newtheorem{definition}{Definition}

\newtheorem{example}{Example}
\newtheorem{remark}{Remark}

\usepackage[all]{xy}

\numberwithin{equation}{section}

\def\I{{{\mathbb I}}}

\def\E{{{\mathbb E }}}

\def\Z{{{\mathbb Z }}}
\def\P{{{\mathbb P }}}

\def\L{{{\mathbb L }}}
\def\N{{{\mathbb N }}}

\def\CC{{{\mathcal C}}}

\def\EE{{{\mathcal E}}}
\def\FF{{{\mathcal F}}}

\def\LL{{{\mathcal L}}}
\def\MM{{{\mathcal M}}}
\def\NN{{{\mathcal N}}}
\def\OO{{{\mathcal O}}}

\def\WW{{{\mathcal W}}}

\def\EExt{{{\mathcal E}xt}}
\def\TTor{{{\mathcal T}or}}
\def\FFitt{{{\mathcal F}itt}}
\def\Hom{{{\rm Hom }}}

\def\Spec{{{\rm Spec \,}}}

\def\Proj{{{\rm Proj \,}}}
\def\Hilb{{{\rm Hilb \,}}}
\def\Univ{{{\rm Univ \,}}}

\def\Quot{{{\rm Quot \,}}}
\def\Supp{{{\rm Supp \,}}}
\def\Pic{{{\rm Pic \,}}}

\def\dim{{{\rm dim \,}}}

\def\hd{{{\rm hd \,}}}
\def\coker{{{\rm coker \,}}}
\def\ker{{{\rm ker \,}}}
\def\rank{{{\rm rank \,}}}
\def\im{{{\rm im \,}}}
\def\id{{{\rm id }}}

\def\length{{{\rm length \,}}}

\def\tors{{{\rm tors \,}}}

\def\red{{{\rm red}}}
\def\Ass{{{\rm Ass\,}}}

\begin{document}
\renewcommand{\refname}{References}
\renewcommand{\proofname}{Proof.}
\thispagestyle{empty}

\title[Morphism of compactifications of moduli of
vector bundles]{On a morphism of compactifications of moduli
scheme of vector bundles}
\author{{N.V.Timofeeva}}%
\address{Nadezda Vladimirovna Timofeeva
\newline\hphantom{iii} Yaroslavl State University,
\newline\hphantom{iii} ul. Sovetskaya, 14,
\newline\hphantom{iii} 150000, Yaroslavl, Russia}%
\email{ntimofeeva@list.ru}%
\vspace{1cm} \maketitle {\small
\begin{quote}
\noindent{\sc Abstract. } A morphism of nonreduced Gieseker --
Maruyama  functor (of semistable coherent torsion-free sheaves) on
the surface
 to the non\-reduced
 functor   of admissible semistable pairs with the same
Hilbert polynomial, is constructed. This leads to the morphism of
moduli schemes with possibly non\-reduc\-ed scheme struct\-ures.
As usually, we study sub\-functors corresponding to main
components of moduli schemes.\medskip

 \end{quote}

\begin{flushright}
{\it To the blessed memory of my Mum}
\end{flushright}

\section*{Introduction}
The purpose of the present paper is to construct a morphism of
main components of Gieseker -- Maruyama  moduli scheme $\overline
M$ of semistable (in the sense of Gieseker -- cf. sect. 1,
definition \ref{defgies}) torsion-free coherent sheaves of fixed
rank and with fixed Hilbert polynomial on a smooth projective
surface, to main components of the moduli scheme $\widetilde M$ of
semistable admissible pairs (cf. sect. 1, definitions \ref{admsch}
and \ref{sstpair}) with same rank and Hilbert polynomial, which
were built up in the series of papers of the author \cite{Tim1} -
\cite{Tim6}. In \cite{Tim4} the construction of a morphism
$\kappa_{\red}: \overline M_{\red} \to \widetilde M_{\red}$ of
same schemes was done but both schemes were considered with
reduced scheme structures. This restriction (the absence of
nilpotent elements in structure sheaves) is essential for the
construction performed in the cited paper. In the present article
we remove this restriction and prove the existence of a morphism
$\kappa: \overline M \to \widetilde M$. The morphism
$\kappa_{\red}$ from \cite{Tim4} is the reduction of $\kappa$ (in
the cited paper the morphism $\kappa_{\red}$ was denoted by
$\kappa$, $\overline M_{\red}$ and $\widetilde M_{\red}$ by
$\overline M$ and $\widetilde M$ respectively). In this way we
give an answer to the open question announced in \cite[remark
3]{Tim6}.

We work on a smooth irreducible projective algebraic surface $S$
over a field  $k=\overline k$ of characteristic zero. On $S$ an
ample invertible sheaf  $L$ is chosen and fixed. It is used as a
polarization. The Hilbert polynomial of the coherent sheaf $E$ of
$\OO_S$-modules having rank $r$, is denoted as $rp(n)$ and is
defined by the formula $rp(n)= \chi(E \otimes L^n)$. Reduced
Hilbert polynomial $p(n)$ of the sheaf $E$ is a polynomial with
rational coefficients. These coefficients depend on the geometry
of the surface $S,$ the polarization $L$ and on Chern classes
$c_1, c_2$ of the sheaf $E$.

The scheme of moduli of Gieseker -- Maruyama $\overline M$ being a
Noetherian projective algebraic scheme of finite type
\cite{Gies},\cite{Mar}, is a classical way to compactify the
moduli space $M_0$ of stable vector bundles with same rank and
Hilbert polynomial.  The scheme $M_0$ is reduced and
quasi-projective (it is a quasi-projective algebraic variety --
\cite{Mar}). In the construction of the scheme $\overline M$
families of locally free stable sheaves are completed by
(possibly, nonlocally free) semistable coherent torsion-free
sheaves with same rank and Hilbert polynomial on the surface $S$
with polarization $L$.

In cited papers \cite{Tim1, Tim2, Tim4} we developed a procedure
to transform a flat family $\E_T$ of semistable coherent
torsion-free sheaves on the surface $S$, parametrized by
irreducible and reduced scheme $T$, to the family
$((\pi:\widetilde \Sigma \to \widetilde T, \widetilde \L),
\widetilde \E)$ of semistable admissible pairs, parameterized by
reduced irreducible scheme $\widetilde T$ which is birational to
$T$. Since we are interested namely in compactifications of moduli
space of stable vector bundles, the family $\E_T$ is thought to
contain at least one locally free sheaf. Since the requirement of
local freeness (as well as requirement of Gieseker-stability --
cf. \cite[propos. 2.3.1]{HL}) are open in flat families, the base
scheme $T$ contains an open subscheme $T_0$ whose closed points
correspond to locally free sheaves. In the compactification we
built, nonlocally free sheaves in points of closed subscheme $T
\setminus T_0$ are replaced by pairs $((\widetilde S, \widetilde
L), \widetilde E)$ in points of closed subscheme $\widetilde T
\setminus \widetilde T_0$, $\widetilde T_0 \cong T_0$, where
$\widetilde S$ is projective algebraic scheme of certain form with
appropriate polarization $\widetilde L,$ and $\widetilde E$ is
appropriate locally free sheaf. As shown in \cite{Tim3}, schemes
$\widetilde S$ are connected, and hence rank $r$ of a locally free
sheaf $\widetilde E$ on such scheme is well-defined. It is equal
to the rank of restriction of the sheaf $\widetilde E$ to each
component of $\widetilde S.$ The Hilbert polynomial of
$\OO_{\widetilde S}$-sheaf $\widetilde E$ is defined in usual way:
$rp(n)=\chi(\widetilde E \otimes \widetilde L^n)$. The precise
description of pairs $((\widetilde S, \widetilde L), \widetilde
E)$ will be given below, sect.1, definitions \ref{admsch} and
\ref{sstpair}.

The mentioned procedure of transformation of a flat family of
semistable torsion-free sheaves to a flat family of admissible
semistable pairs is called a standard resolution. It gives rise to
a birational morphism of base schemes $\widetilde T \to T$ which
becomes an isomorphism when restricted to the preimage $\widetilde
T_0$ of open subscheme $T_0$ of locally free sheaves. To perform a
standard resolution as it developed in articles \cite{Tim1, Tim2,
Tim4}, one needs irreducibility and reducedness of the base scheme
$T$.

It is known \cite{O'Gr, O'Gr1, Gies-Li} that for arbitrary surface
$S$ the scheme $\overline M$ is asymptotically (in particular, for
big values of $c_2$) reduced, irreducible and of expected
dimension. Although under arbitrary choice of numerical invariants
of sheaves this scheme can be nonreduced.

In this article  we develop 
the version of the standard resolution for the family of
semistable coherent torsion-free sheaves for the case when the
base scheme is nonreduced. For our considerations it is enough to
restrict ourselves by the class of schemes $T$ such that their
reductions $T_{\red}$ are irreducible schemes.

Using our version of standard resolution we construct the natural
transformation of the Gieseker -- Maruyama functor (sect. 1,
(\ref{funcGM}),(\ref{famGM})) to the functor $\mathfrak f$ of
admissible semistable pairs (sect. 1, (\ref{class}),
(\ref{funcmy})). The natural transformation leads to the morphism
of moduli schemes.

The article consists of four sections. In sect. 1 we recall the
basic notions which are necessary for the further considerations.
Therein we give a standard description how the morphism of moduli
functors determines morphism of their moduli schemes. Sect. 2 is
devoted to procedure of standard resolution of a family of
coherent sheaves with non-reduced base. Sect. 3 contains the
construction of the natural transformation ${\mathfrak f}^{GM}\to
{\mathfrak f}$ using standard resolution from sect. 2. In
addition, in sect. 4 we obtain the morphism of moduli schemes of
functors of our interest induced by the natural transformation
${\mathfrak f}^{GM} \to {\mathfrak f}$, "by hands", without
category-theoretical constructions.

In the present article we prove the following result.
\begin{theorem} The Gieseker -- Maruyama functor
${\mathfrak f}^{GM}$ of semistable torsion-free coherent sheaves
of rank $r$ and with Hilbert polynomial $rp(n)$ on the surface
$(S,L)$, has a natural transformation to the functor ${\mathfrak
f}$ of admissible semistable pairs $((\widetilde S,\widetilde L),
\widetilde E)$ where the locally free sheaf $\widetilde E$ on the
projective scheme  $(\widetilde S, \widetilde L)$ has same rank
and Hilbert polynomial. In particular, there is a morphism of
moduli schemes $\overline M \to \widetilde M$ associated with this
natural transformation.
\end{theorem}

{\it Acknowledgement.} The author is cordially grateful to V.S.
Kulikov and V.V.Sho\-kurov (V.A. Steklov Institute, Moscow) for
lively interest to the work.

\section{Objects and functors}
Throughout  in this paper we identify locally free $\OO_X$-sheaf
on a scheme $X$ with corresponding vector bundle and use both
terms as synonyms.

We use the classical definition of semistability due to Gieseker
\cite{Gies}.
\begin{definition} \label{defgies} Coherent $\OO_S$-sheaf
$E$ is {\it stable} (resp., {\it semistable}) if for any proper
subsheaf  $F\subset E$ of rank  $r'=\rank F$ for  $n\gg 0$
$$
\frac{\chi(E\otimes L^n)}{r}>\frac{\chi(F\otimes L^n)}{r'},\;\;
({\mbox{\rm resp.,}} \;\; \frac{\chi(E\otimes L^n)}{r}\ge
\frac{\chi(F\otimes L^n)}{r'}\;).
$$
\end{definition}

Consider the Gieseker -- Maruyama functor
\begin{equation}\label{funcGM}{\mathfrak f}^{GM}:
(Schemes_k)^o \to Sets \end{equation} attaching to any scheme $T$
the set of equivalence classes of families of the form ${\mathfrak
F}_T^{GM}/\sim$ where
\begin{equation}\label{famGM} \mathfrak
F_T^{\,GM}= \left\{
\begin{array}{l} \E \mbox{\rm -- sheaf of } \OO_{T\times S}-
\mbox{\rm modules, flat over } T;\\
\L \mbox{\rm -- invertible sheaf of } \OO_{T\times S}-\mbox{\rm
modules,
 ample relative to } T\\
\mbox{\rm and such that } L|_{t\times S}\cong L\; \mbox{\rm for any } t\in T;\\
E_t:=\E|_{t\times S} \mbox{\rm -- torsion-free and semistable
 due to Gieseker}
;\\
\chi(E_t \otimes L_t^n)=rp(n).\end{array}\right\}
\end{equation}

The equivalence relation $\sim$ is defined as follows. Families
$(\E, \L)$ and $(\E',\L')$ from the class $\mathfrak F^{GM}_T$ are
said to be equivalent (notation: $ (\E, \L)\sim (\E',\L')$) if
there are line bundles  $L', L''$ on the scheme $T$ such that $\E'
= \E \otimes pr_1^{\ast} L',$
 $ \L' =  \L \otimes pr_1^{\ast} L''$ where $pr_1: T\times S \to
T$ is projection to the first factor.

\begin{remark} Since $\Pic (T \times S) =\Pic T \times \Pic
S$, our definition for moduli functor ${\mathfrak f}^{GM}$ is
equivalent to the standard definition as formulated, for example,
in \cite{HL}: the difference in the choice of polarizations $\L$
and $\L'$ with isomorphic restrictions on fibres over $T$, is
eliminated by tensoring by the inverse image of appropriate
invertible sheaf $L''$ from the base $T$.
\end{remark}

To proceed further, recall the definition of a sheaf of zeroth
Fitting ideals which is known from commutative algebra \cite[ch.
III, sect. 20.2]{Eisen}. Let $X$ be a scheme,  $F$ --
$\OO_X$-module with finite presentation $F_1
\stackrel{\varphi}{\longrightarrow} F_0 \to F$. Here
$\OO_X$-modules $F_0$ and $F_1$ are assumed to be locally free.
Without loss of generality we suppose that $\rank F_1 \ge \rank
F_0$.
\begin{definition} {\it The sheaf of zeroth Fitting ideals} for
$\OO_X$-module $F$ is defined as  $\FFitt^0 F =\im
(\bigwedge^{\rank F_0} F_1 \otimes \bigwedge^{\rank F_0}
F_0^{\vee} \stackrel{\varphi'}{\longrightarrow}\OO_X)$, where
$\varphi'$ is a morphism of $\OO_X$-modules induced by the
morphism $\varphi$.
\end{definition}

\begin{definition} \label{admsch} \cite{Tim3, Tim4} Polarized algebraic scheme
$(\widetilde S, \widetilde L)$ is {\it admissible} if the scheme
$(\widetilde S,\widetilde L)$ satisfies one of the conditions

 i)
$(\widetilde S, \widetilde L) \cong (S,L)$,

ii) $\widetilde S \cong {\rm Proj \,} \bigoplus_{s\ge
0}(I[t]+(t))^s/(t^{s+1})$, where $I={{\mathcal F}itt}^0 {{\mathcal
E}xt}^2(\varkappa, {\mathcal O}_S)$ for Artinian quotient sheaf
$q: \bigoplus^r {\mathcal O}_S\twoheadrightarrow \varkappa$ of
length  $l(\varkappa)\le c_2$. There is a morphism $\sigma:
\widetilde S \to S$ (which is called {\it canonical morphism}) and
$\widetilde L = L \otimes (\sigma ^{-1} I \cdot {\mathcal
O}_{\widetilde S})$
--- ample invertible sheaf on  $\widetilde S$; this polarization
$\widetilde L$ is called the {\it distinguished} polarization.
\end{definition}

\begin{remark} We require the sheaf $\widetilde L = L \otimes
(\sigma ^{-1} I \cdot {\mathcal O}_{\widetilde S})$ to be ample on
the scheme $\widetilde S$. If it is not true, replace
 $\OO_S$-sheaf $L$ by its big enough tensor power (which can be chosen
common for all $\widetilde S$, as shown in \cite[claim 1]{Tim6}).
We redenote this tensor power again as $L$.
\end{remark}

\begin{remark} The canonical morphism $\sigma: \widetilde S \to S$
is determined by the structure of $\OO_S$-algebra on
$\bigoplus_{s\ge 0}(I[t]+(t))^s/(t^{s+1}).$
\end{remark}

As shown in \cite[Introduction]{Tim6}, the scheme $\widetilde S$
consists of several components unless it is isomorphic to $S$. The
main component  $\widetilde S_0$ is the initial surface $S$ blown
up in the sheaf of ideals $I$. The ideal $I$ defines some
zero-dimensional subscheme in $S$ with structure sheaf $\OO_X/I.$
Main component $\widetilde S_0$ corresponds to an algebraic
variety which can be singular. Additional components $\widetilde
S_i$, $i>0$ can carry  nonreduced scheme structures. Number of
additional components, i.e. number of components of the scheme
$\overline{\widetilde S \setminus \widetilde S_0}$, equals to the
number of maximal ideals of the quotient algebra $\OO_S/I$.

\begin{definition} \label{sstpair}  \cite{Tim4,Tim6}
$S$-{\it stable} (resp.,{\it semistable}) {\it pair} $((\widetilde
S,\widetilde L), \widetilde E)$ is the following data:
\begin{itemize}
\item{$\widetilde S=\bigcup_{i\ge 0} \widetilde S_i$ --
admissible scheme, $\sigma: \widetilde S \to S$ --- its canonical
morphism, $\sigma_i: \widetilde S_i \to S$
--- restrictions of $\sigma$ to components $\widetilde S_i$,
$i\ge 0;$}
\item{$\widetilde E$ --- vector bundle on the scheme
$\widetilde S$;}
\item{$\widetilde L \in \Pic \widetilde S$ --- distinguished polarization;}
\end{itemize}
such that
\begin{itemize}
\item{$\chi (\widetilde E \otimes \widetilde
L^n)=rp(n),$ the polynomial $p(n)$ and the rank $r$ of the sheaf
$\widetilde E$ are fixed;}
\item{on the scheme $\widetilde S$ the sheaf
$\widetilde E$ is {\it stable} (resp., {\it semistable}) {\it due
to Gieseker} i.e. for any proper subsheaf $\widetilde F \subset
\widetilde E$ for $n\gg 0$
\begin{eqnarray*}
\frac{h^0(\widetilde F\otimes \widetilde L^n)}{{\rm rank\,} F}&<&
\frac{h^0(\widetilde E\otimes \widetilde L^n)}{{\rm rank\,} E},
\\ (\mbox{\rm resp.,} \;\;
\frac{h^0(\widetilde F\otimes \widetilde L^n)}{{\rm rank\,}
F}&\leq& \frac{h^0(\widetilde E\otimes \widetilde L^n)}{{\rm
rank\,} E}\;);
\end{eqnarray*}}
\item{on any of additional components $\widetilde S_i, i>0,$
the sheaf $\widetilde E_i:=\widetilde E|_{\widetilde S_i}$ is {\it
quasi-ideal,} i.e. it has a description
\begin{equation}\label{quasiideal}\widetilde E_i=
\sigma_i^{\ast}\ker q_0/\tors_i.\end{equation} for some $q_0\in
\bigsqcup_{l\le c_2} {\rm Quot\,}^l \bigoplus^r {\mathcal O}_S$.
The epimorphism $q_0: \bigoplus^r \OO_S \twoheadrightarrow
\varkappa$ is common for all components $\widetilde S_i$ of the
scheme $\widetilde S$ as well as $l=\length \varkappa.$
}\end{itemize}
\end{definition}

Subsheaf  $\tors_i$ is defined as a restriction
$\tors_i=\tors|_{\widetilde S_i}$ where the role of $\tors$ in our
considerations is analogous to the role of torsion subsheaf on
reduced scheme. Let $U$ be Zariski open subset in one of
components $\widetilde S_i, i\ge 0$ and $\sigma^{\ast}\ker
q_0|_{\widetilde S_i}(U)$ be the corresponding group of sections
carrying structure of ${\mathcal O}_{\widetilde S_i}(U)$-module.
Sections $s\in \sigma^{\ast}\ker q_0|_{\widetilde S_i}(U)$
annihilated by prime ideals of positive codimensions in
 ${\mathcal O}_{\widetilde S_i}(U)$, form a submodule in
$\sigma^{\ast}\ker q_0|_{\widetilde S_i}(U)$. We denote this
submodule as $\tors_i(U)$. The correspondence $U \mapsto
\tors_i(U)$ defines a subsheaf $\tors_i \subset \sigma^{\ast}\ker
q_0|_{\widetilde S_i}.$ Note that associated primes of positive
codimensions annihilating sections  $s\in \sigma^{\ast}\ker
q_0|_{\widetilde S_i}(U)$, correspond to subschemes supported in
the preimage  $\sigma^{-1}({\rm Supp\,}
\varkappa)=\bigcup_{i>0}\widetilde S_i.$ Since by the construction
the scheme $\widetilde S=\bigcup_{i\ge 0}\widetilde S_i$ is
connected, subsheaves  $\tors_i, i> 0,$ allow to form a subsheaf
$\tors \subset \sigma^{\ast}\ker q_0$. The former is defined as
follows. Section $s\in \sigma^{\ast}\ker q_0|_{\widetilde S_i}(U)$
satisfies  $s\in \tors|_{\widetilde S_i}(U)$ if and only if
\begin{itemize}
\item{there exists a section $y\in {\mathcal O}_{\widetilde S_i}(U)$ such that $ys=0$,}
\item{at least one of following conditions is satisfied: either $y\in {\mathfrak p}$ where
$\mathfrak p$ is prime ideal of positive codimension, or there
exist Zariski-open subset  $V\subset \widetilde S$ and a section
$s' \in \sigma^{\ast}\ker q_0 (V)$ such that $V\supset U$,
$s'|_U=s$, and $s'|_{V\cap \widetilde S_0} \in \tors
(\sigma^{\ast}\ker q_0|_{\widetilde S_0})(V\cap \widetilde S_0)$.
In the former expression the subsheaf of torsion
$\tors(\sigma^{\ast}\ker q_0|_{\widetilde S_0})$ is understood in
usual sense.}
\end{itemize}

As shown in \cite{Tim4}, there is a map taking any semistable
coherent torsion-free sheaf $E$ to the admissible semistable pair
$((\widetilde S, \widetilde L), \widetilde E)$ as follows. If $E$
is locally free then $((\widetilde S, \widetilde L), \widetilde
E)=((S,L),E)$. Otherwise $\widetilde S=\Proj \bigoplus_{s\ge
0}(I[t]+(t))^s/(t^{s+1})$, where $I={{\mathcal F}itt}^0 {{\mathcal
E}xt}^1(E, \OO_S)$, $\widetilde L = L \otimes (\sigma ^{-1} I
\cdot \OO_{\widetilde S})$ and $\widetilde E=\sigma^{\ast}
E/\tors$ where $\tors$ is understood as described. This map
corresponds to the morphism $\kappa_{\red}: \overline M_{\red} \to
\widetilde M_{\red}.$

Let $T,S$ be  schemes over a field $k$, $\pi: \widetilde \Sigma
\to T$ a morphism of $k$-schemes. We introduce the following

\begin{definition}\label{bitriv} The family of schemes
$\pi: \widetilde \Sigma \to T$ is {\it birationally $S$-trivial}
if there exist isomorphic open subschemes $\widetilde \Sigma_0
\subset \widetilde \Sigma$ and $\Sigma_0 \subset T\times S$ and
there is a scheme equality $\pi(\widetilde \Sigma_0)=T$.
\end{definition}
The former equality means that all fibres of the morphism $\pi$
have nonempty intersections with the open subscheme $\widetilde
\Sigma_0$.

In particular, if $T=\Spec k$ then $\pi$ is a constant morphism
and $\widetilde \Sigma_0 \cong \Sigma_0$ is open subscheme in $S$.

Since in the present paper we consider only $S$-birationally
trivial families, they will be referred to as {\it birationally
trivial} families.

Also we consider families of semistable pairs
%
\begin{equation}\label{class}{\mathfrak F}_T= \left\{
\begin{array}{l}\pi: \widetilde \Sigma \to T \mbox{\rm \;\;birationally $S$-trivial
},\\
\widetilde \L\in \Pic \widetilde \Sigma \mbox{\rm \;\;flat over
}T,\\
\mbox{\rm for }m\gg 0 \; \widetilde \L^m \; \mbox{\rm very ample
relatively }T,\\ \forall t\in T \;\widetilde L_t=\widetilde
\L|_{\pi^{-1}(t)}
\mbox{\rm \; ample;}\\
(\pi^{-1}(t),\widetilde L_t) \mbox{\rm \;admissible scheme with
distinguished polarization}; \\
\chi (\widetilde L_t^n) \mbox{\rm \; does not depend on }t,\\
 \widetilde \E \;\; \mbox{\rm locally free } \OO_{\Sigma}-\mbox{\rm
sheaf flat over } T;\\
 \chi(\widetilde \E\otimes\widetilde \L^{n})|_{\pi^{-1}(t)})=
 rp(n);\\
 ((\pi^{-1}(t), \widetilde L_t), \widetilde \E|_{\pi^{-1}(t)}) -
 \mbox{\rm semistable pair}
 \end{array} \right\} \end{equation} and a functor
\begin{equation}\label{funcmy}
\mathfrak f: (Schemes_k)^o \to (Sets)\end{equation} from the
category of $k$-schemes to the category of sets. It attaches to a
scheme $T$ the set of equivalence classes of families of the form
${\mathfrak F}_T/\sim$.

The equivalence relation $\sim$ is defined as follows. Families
$((\pi: \widetilde \Sigma \to T, \widetilde \L), \widetilde \E)$
and $((\pi': \widetilde \Sigma \to T, \widetilde \L'), \widetilde
\E')$ from the class $\mathfrak F_T$ are said to be equivalent
(notation: $((\pi: \widetilde \Sigma \to T, \widetilde \L),
\widetilde \E) \sim ((\pi': \widetilde \Sigma \to T,
\widetilde \L'), \widetilde \E')$) if\\
1) there exists an isomorphism  $ \iota: \widetilde \Sigma
\stackrel{\sim}{\longrightarrow} \widetilde \Sigma'$ such that the
diagram \begin{equation*} \xymatrix{\widetilde \Sigma
\ar[rd]_{\pi}\ar[rr]_{\sim}^{\iota}
&&\widetilde \Sigma ' \ar[ld]^{\pi'}\\
&T }
\end{equation*} commutes.\\
2) There exist line bundles $L', L''$ on $T$ such that
$\iota^{\ast}\widetilde \E' = \widetilde \E \otimes \pi^{\ast}
L',$ $\iota^{\ast}\widetilde \L' = \widetilde \L \otimes
\pi^{\ast} L''.$

\begin{remark} The definition of the functor of semistable admissible pairs
given here differs from definition from preceding papers
\cite{Tim4} -- \cite{Tim7}: we added new requirement of birational
triviality. This requirement was missed by the author before.
Indeed, without this require\-ment, the consideration of
"twisted"\, families of schemes is allowed. One can take even
twisted families with all fibres isomorphic to the surface $S$.
This makes sets ${\mathfrak F}_T/\sim$ {\it too big} in such a
sense that non-isomorphic families of schemes over a same base $T$
arise where corresponding fibres are isomorphic. For a simple
example take a projective plane  $S=\P^2$. Not any  $\P^2$-bundle
over the base $T$ is isomorphic to the product  $ T\times \P^2$.
This circumstance does not allow to conclude the isomorphism of
open subfunctor of Gieseker-semistable vector bundles, to an open
subfunctor corresponding to semistable $S$-pairs if the former is
considered without requirement of birational triviality. All the
results of articles \cite{Tim4} -- \cite{Tim7}  become true
together with their proofs when the requirement of birational
triviality is added.
\end{remark}

Now discuss what is the "size"\, of the maximal under  inclusion
of those open sub\-schemes $\widetilde \Sigma_0$ in a family of
admissible schemes $\widetilde \Sigma$, which are isomorphic to
appropriate open subschemes in $T\times S$ in the definition
\ref{bitriv}. The set $F=\widetilde \Sigma \setminus \widetilde
\Sigma_0$ is closed. If $T_0$ is open subscheme in $T$ whose
points carry fibres isomorphic to $S$, then $\widetilde \Sigma_0
\supsetneqq \pi^{-1}T_0$ (inequality is true because $\pi
(\widetilde \Sigma_0)=T$ in the definition \ref{bitriv}). The
subscheme $\Sigma_0$ which is open in $T\times S$ and isomorphic
to $\widetilde \Sigma_0$, is such that $\Sigma_0 \supsetneqq
T_0\times S$. If $\pi:\widetilde \Sigma \to T$ is family of
admissible schemes then $\widetilde \Sigma_0 \cong \widetilde
\Sigma \setminus F$, and $F$ is (set-theoretically) the union of
additional components of fibres which are non-isomorphic to $S$.

Following \cite[ch. 2, sect. 2.2]{HL} we recall some definitions.
Let  $\CC$ be a category, $\CC^o$ its dual category,
 $\CC'={\FF}unct(\CC^o, Sets)$ a category of functors to
 the category of sets. By Yoneda's lemma, the functor $\CC \to \CC':
F\mapsto (\underline F: X\mapsto \Hom_{\CC}(X, F))$ includes
 $\CC$ as full subcategory in $\CC'$.

\begin{definition}\label{corep} \cite[ch. 2, definition 2.2.1]{HL}
The functor  ${\mathfrak f} \in {\OO}b\, \CC'$ is {\it
corepresented by the object} $F \in {\OO}b \,\CC$ if there exists
a  $\CC'$-morphism  $\psi : {\mathfrak f} \to \underline F$ such
that any morphism  $\psi': {\mathfrak f} \to \underline F'$
factors through the unique morphism $\omega: \underline F \to
\underline F'$.
\end{definition}
 \begin{definition} The scheme $ M$ is a {\it coarse moduli space
 } for the functor $\mathfrak f$ if  $\mathfrak f$ is corepresented by
 $M.$
 \end{definition}

Now let we are given two functors ${\mathfrak f},
\widetilde{\mathfrak f}: \CC^o \to Sets$, and a natural
transformation $\underline \kappa: {\mathfrak f} \to
\widetilde{\mathfrak f}$ where for any  $T\in {\OO}b \, \CC$ there
is a commutative diagram
\begin{equation*}
\xymatrix{T \ar[d]_{=} \ar@{|->}[r]^{\mathfrak
f}&{\mathfrak f}(T) \ar[d]^{\underline \kappa(T)}\\
T \ar@{|->}[r]^{\widetilde{\mathfrak f}}&\widetilde{\mathfrak
f}(T)}
\end{equation*}
Let the functor ${\mathfrak f}$ have coarse moduli space $M$ and
the functor  $\widetilde{\mathfrak f}$ have coarse moduli space
$\widetilde M$. In the situation if definition \ref{corep} the
following diagram commutes:
\begin{equation*}\xymatrix{T \ar@{|->}[d]\ar[rr]^{=}&&T \ar@{|->}[d]\\
{\mathfrak f}(T) \ar[dd]_{\psi'(T)} \ar[dr]^{\psi(T)}
\ar[rr]^{\underline\kappa(T)}&& \widetilde{\mathfrak
f}(T) \ar[dd]_>>>>>>>{\widetilde \psi'(T)} \ar[dr]^{\widetilde \psi(T)}\\
&\ar[ld]^{\omega(T)}\Hom(T,M) \ar@{..>}[rr]^<<<<<<<<{\epsilon}&&
\ar[ld]^{\widetilde \omega(T)}\Hom(T, \widetilde M)\\
\Hom(T,F')\ar[rr]&&\Hom(T, F')}
\end{equation*}
The dash arrow  $\epsilon$ is defined since the functor
${\mathfrak f}$ is corepresented by the object $M$. Setting  $T=M$
we have  $\kappa:=\epsilon(\id): M\to \widetilde M$.

In this article  $\overline M$ is moduli space for  ${\mathfrak
f}^{GM}$,  and  $\widetilde M$ is moduli space for ${\mathfrak
f}$. We construct a natural transformation of functors $\mathfrak
f^{GM} \to \mathfrak f$, expressed by the diagram
\begin{equation*}\xymatrix{T \ar@{=}[d]\ar
@{|->}[r]^{\mathfrak f^{GM}}&{\mathfrak F}^{GM}_T/\sim \ar[d]\\
 T \ar@{|->}[r]^{\mathfrak f}&{\mathfrak
F}_{T}/\sim}
\end{equation*}
This natural transformation yields  a morphism of moduli schemes
$\kappa: \overline M \to \widetilde M$.

\section{Standard resolution for families with nonreduced base}

In this section we develop the analogue of the procedure of
standard resolution from articles  \cite{Tim1} -- \cite{Tim4}.

Let $T$ be arbitrary  (possibly nonreduced) $k$-scheme of finite
type. We assume that its reduction  $T_{\red}$ is irreducible. If
$\E$ is a family of coherent torsion-free sheaves on the surface
$S$ having nonreduced base $T$ then homological dimension of  $\E$
as $\OO_{T\times S}$-module is not greater then 1. The proof of
this fact for reduced equidimensional base can be found, for
example, in \cite[proposition 1]{Tim}.

Now we need the following simple lemma concerning homological
dimension of the family $\E$ with nonreduced base.
\begin{lemma}
Let coherent $\OO_{T\times S}$-module $\E$ of finite type is
$T$-flat  and its reduction $\E_{\red}:=\E \otimes_{\OO_{T}}
\OO_{T_{\red}}$ has homological dimension not greater then 1:
$\hd_{T_{\red}\times S}\E_{\red} \le 1.$ Then $\hd_{T\times S}
\E\le 1$.
\end{lemma}
\begin{proof}
Assume the opposite; let in the exact triple of $\OO_{T\times
S}$-modules
$$
0\to E_1 \to E_0 \to \E \to 0
$$
 $E_1$ be nonlocally free as $\OO_{T\times S}$-module while $E_0$
 is locally free. Passing to reductions (this means tensoring by
$\otimes_{\OO_T} \OO_{T_{\red}}$) we come to the exact triple
$$0\to E_{1\red} \to E_{0\red} \to \E_{\red}\to 0.
$$
Here we took into account that since $\E$ is $T$-flat, then
$\TTor_1^{\OO_T}(\E, \OO_{T_{\red}})=0$, and left-exactness is
preserved. Since $\hd_{T_{\red}\times S} \E_{\red}\le 1,$ then
$E_{1\red}$ is locally free as $\OO_{T_{\red}\times S}$-module.
Also since $\E$ and $E_0$ are $T$-flat and of finite type then
$E_1$ is also $T$-flat and of finite type. It rests to conclude
that $E_1$ is locally free.   Apply the sheaf-theoretic version of
the following result from  Grothendieck's SGA \cite[ch. IV,
corollaire 5.9]{SGAI}:
\begin{proposition} If $X\stackrel{g}{\to} Y\stackrel{f}
{\to} T$ are morphisms of Noetherian schemes and $f$ is flat
morphism then a coherent $\OO_X$-sheaf $\EE$ is flat over $Y$ if
and only if it is flat over $T$ and for any closed point $t\in T$
$\EE|_{(f\circ g)^{-1}(t)}$ is flat over $\OO_{f^{-1}(t)}$.
\end{proposition}

Set $X=Y=T\times S$, $g: X\to Y$ to be an identity isomorphism,
$\EE:=E_1$, $f=p: T\times S \to T$  a projection,
$p^{-1}(t)=t\times S$. Then for any closed point $t\in T$ (keeping
in mind that sets of closed points are equal and residue fields of
corresponding closed points are isomorphic for a scheme $T$ and
for its reduction $T_{\red}$) $E_1|_{t\times
S}=E_{1\red}|_{t\times S}$, and $E_{1\red}|_{t\times S}$ is flat
over $\OO_{t\times S}$ because of local freeness. From this we
conclude that $E_1$ is locally free as $\OO_{T\times S}$-module.
This completes the proof of the lemma.
\end{proof}

We do computations as in \cite{Tim1}. Choose and fix locally free
$\OO_{T\times S}$-resolution of the sheaf $\E$:
\begin{equation}\label{lfr}
0\to E_1 \to E_0 \to \E \to 0.
\end{equation}
Form a sheaf of 0-th Fitting ideals
\begin{equation}\label{fitt}\I=\FFitt^0 \EExt^1(\E, \OO_{T\times
S}).\end{equation} Let  $\sigma \!\!\! \sigma: \widehat{\Sigma}\to
\Sigma$ be a morphism of blowing up the scheme $\Sigma:=T\times S$
in the sheaf of ideals  $\I$. Apply inverse image $\sigma \!\!\!
\sigma^{\ast}$ to the dual sequence of (\ref{lfr}):
\begin{eqnarray}\xymatrix{
\sigma \!\!\! \sigma^{\ast}\E^{\vee}\ar[r]&\sigma \!\!\!
\sigma^{\ast}E_0^{\vee}\ar[r]
&\sigma \!\!\! \sigma^{\ast}\WW \ar[r]&0, } \label{cc} \nonumber\\
\label{c}\xymatrix{ \sigma \!\!\! \sigma^{\ast}\WW\ar[r]&\sigma
\!\!\! \sigma^{\ast}E_1^{\vee}\ar[r] &\sigma \!\!\!
\sigma^{\ast}\EExt^1 (\E, \OO_{\Sigma})\ar[r]&0.}
\end{eqnarray}
The symbol  $\WW$ stands for the sheaf  $$\ker(E_1^{\vee} \to
\EExt^1(\E, \OO_{\Sigma}))=\coker(\E^{\vee} \to E_0^{\vee}).$$

In  (\ref{c}) denote $\NN:=\ker(\sigma \!\!\!
\sigma^{\ast}E_1^{\vee}\to \sigma \!\!\! \sigma^{\ast}\EExt^1 (\E,
\OO_{\Sigma}))$. The sheaf   $\FFitt^0(\sigma \!\!\!
\sigma^{\ast}\EExt^1(\E, \OO_{\Sigma}))$ is invertible by
functorial property of  $\FFitt$:
\begin{eqnarray*}\FFitt ^0(\sigma \!\!\! \sigma
^{\ast}\EExt^1_{\OO_{\Sigma}}(\E,\OO_{\Sigma}))= (\sigma \!\!\!
\sigma^{-1}\FFitt
^0(\EExt^1_{\OO_{\Sigma}}(\E,\OO_{\Sigma})))\cdot
\OO_{\widehat{\Sigma}}= (\sigma \!\!\! \sigma^{-1}\I)\cdot
\OO_{\widehat{\Sigma}}.
\end{eqnarray*}
Although 0th Fitting ideal sheaf is invertible, non-reducedness of
the scheme $\widehat \Sigma$ makes \linebreak Tikhomirov's lemma
\cite[lemma 1]{Tikh} nonapplicable in its initial formulation.
However it can be slightly generalized as follows.

\begin{lemma} Let  $X$ be Noetherian scheme such that its reduction
$X_{\red}$ is irreducible, $\FF$ nonzero coherent $\OO_X$-sheaf
supported on a subscheme of codimension $\ge 1$. Then the sheaf of
0-th Fitting ideals $\FFitt^0(\FF)$ is invertible $\OO_X$-sheaf if
and only if $\FF$ has homological dimension equal to 1: $\hd_X
\FF=1.$
\end{lemma}
\begin{proof} is almost identical to the proof in
\cite[Lemma 1]{Tikh}, except some details. The part  "if"\, is
obvious. For opposite implication it is necessary to prove that
$\hd_{\OO_{X,x}}\FF_x=1$ for any point $x\in X$. Irreducibility of
reduction  $X_{\red}$ means that the local ring $\OO_{X,x}$
contains no zero-divisors except nilpotents.

Consider finite presentation of $\OO_{X,x}$-module $\FF_x$:
\begin{equation} \label{pres}
M \stackrel{f}{\to} N \to \FF_x \to 0,
\end{equation}
where  $M$ and $N$ are free  $\OO_{X,x}$-modules of the form
$M=\bigoplus_{i=1}^{n+r} \OO_{X,x} e_i$ and \linebreak
$N=\bigoplus_{j=1}^{n} \OO_{X,x} e'_j$ respectively. Let
$A=(a_{ij})$ be a matrix of $\OO_{X,x}$-linear map $f$ with
respect to systems of generating elements $(e_i)_{i=1}^{n+r}$,
$(e'_j)_{j=1}^n$ where $a_{ij}\in \OO_{X,x}$. Since  $Fitt^0
\FF_x$ is principal ideal in $\OO_{X,x}$ generated by all $n\times
n$-minors of the matrix $A$, we can put (possibly after
re-ordering of generating elements) that $Fitt^0 \FF_x=(a)$,
$$
a=\det \left[\begin{array}{ccc} a_{11}& \dots &
a_{1n}\\
\vdots&\ddots&\vdots\\
a_{n1}& \dots & a_{nn}\\\end{array} \right] \in \OO_{X,x}.$$ The
element  $a$ is non-zerodivizor in  $\OO_{X,x}$, and by our
restriction on the form of the ring $\OO_{X,x}$ this element is
not nilpotent.

Assume that  $r\ge 1$ (the case  $r=0$ will be considered
separately). For $1 \le i \le n$, $n+1 \le k \le n+r$ denote by
$\Delta_{ik}$ $n\times n$-minor of the matrix $A$. This minor is
obtained by replacing of  $i$-th column by $k$-th column of the
matrix  $A$. For  $n+1 \le k \le n+r$ consider  $r$ similar
systems each of $n$ linear equations (one system for each $k$) in
$n$ variables:
$$\sum_{j=1}^n a_{ij} x_{jk}=a_{ik}, \quad 1\le i\le n,
\quad n+1\le k\le n+r.$$ By Cramer's rule following relations are
true \begin{equation} \label{Cramer}\sum_{j=1}^n a_{ij}
\Delta_{jk}=a_{ik} a, \quad 1\le i\le n, \quad n+1\le k\le
n+r.\end{equation} Since ideal  $Fitt^0 \FF_x$ is 1-generated by
the element  $a$, there exist such  $\lambda_{ik} \in \OO_{X,x},$
that $\Delta_{ik}=\lambda_{ik} a$. Their substitution  in
(\ref{Cramer}) yields \begin{equation*} (\sum_{j=1}^n a_{ij}
\lambda_{jk}) a=a_{ik} a, \quad 1\le i\le n, \quad n+1\le k\le
n+r.\end{equation*} Since  $a$ is non-zerodivisor, then
$(\sum_{j=1}^n a_{ij} \lambda_{jk}) =a_{ik} , \quad 1\le i\le n,
\quad n+1\le k\le n+r.$ By the definition of the matrix $A$ this
means that  $f(e_k) \in f(\bigoplus_{i=1}^n \OO_{X,x} e_i)$ for
$n+1 \le k \le n+r.$ Hence  $f(M)=f(\bigoplus_{i=1}^n \OO_{X,x}
e_i)$. Then replacing in  (\ref{pres}) $M$ by $\bigoplus_{i=1}^n
\OO_{X,x}e_i$ we have $r=0$.

It rests to prove injectivity of the homomorphism $f$ for $r=0.$
Consider the system of linear equations $$ \sum_{i=1}^n
a_{li}x_i=0,
$$
defining $\ker f$. Denoting by $X$ the column of indeterminates
$x_1,\dots, x_n$ we obtain a matrix equation. Let  $A^{\ast}$ be
an adjoint matrix for  $A$. Elements $a^{\ast}_{kl}$ of $A^{\ast}$
equal to algebraic complements of elements of $A$:
$a^{\ast}_{kl}=A_{lk}$. Then  $A^{\ast} A=E\det A $ and hence
$A^{\ast}A X=(\det A) X=aX=0$. Since  $a$ is non-zerodivisor, then
$X=0$. This completes the proof of the lemma.
 \end{proof}

Applying the lemma we conclude that    $\hd \sigma \!\!\! \sigma
^{\ast}\EExt^1_{\OO_{\Sigma}}(\E,\OO_{\Sigma})=1.$

Hence the sheaf  $\NN=\ker(\sigma \!\!\!
\sigma^{\ast}E_1^{\vee}\to \sigma \!\!\! \sigma
^{\ast}\EExt^1_{\OO_{\Sigma}}(\E,\OO_{\Sigma}) $ is locally free.
Then there is a morphism of locally free sheaves $\sigma\!\!\!
\sigma^{\ast}E_0^{\vee} \to \NN.$ Let  $Q$ be a sheaf of
$\OO_{\Sigma}$-modules which factors the morphism
 $E_0^{\vee} \to
E_1^{\vee}$ into the composite of epimorphism and monomorphism. By
the definition of the sheaf $\NN$ it also factors the morphism
$\sigma\!\!\!\sigma^{\ast}Q \to
\sigma\!\!\!\sigma^{\ast}E_1^{\vee}$ in the composite of
epimorphism and monomorphism and
$\sigma\!\!\!\sigma^{\ast}E_0^{\vee}\to
\sigma\!\!\!\sigma^{\ast}Q$ is an epimorphism. From this we
conclude that the composite
$\sigma\!\!\!\sigma^{\ast}E_0^{\vee}\to
\sigma\!\!\!\sigma^{\ast}Q\to \NN$ is an epimorphism of locally
free sheaves. Then its kernel is also locally free sheaf. Now set
$\widehat \E:=\ker(\sigma\!\!\! \sigma^{\ast}E_0^{\vee} \to
\NN)^{\vee}$. Consequently we have an exact triple of locally free
$\OO_{\widehat \Sigma}$-modules $$ 0\to \widehat \E^{\vee} \to
\sigma\!\!\!\sigma^{\ast}E_0^{\vee} \to \NN \to 0.
$$ Its dual is also exact.

Now there is a commutative diagram with exact rows
\begin{equation}\label{depi}\xymatrix{0\ar[r]& \NN^{\vee}\ar[r]&\sigma\!\!\!\sigma^{\ast}E_0 \ar[r]&\widehat \E \ar[r]&0\\
& \sigma\!\!\!\sigma^{\ast}E_1 \ar[u]
\ar[r]&\sigma\!\!\!\sigma^{\ast}E_0 \ar@{=}[u] \ar[r]&
\sigma\!\!\!\sigma^{\ast}\E \ar[u] \ar[r]&0}
\end{equation}
where the right vertical arrow is an epimorphism.

Recall the following definition from \cite[$O_{III}$, definition
9.1.1]{EGAIII}.
\begin{definition} The continuous mapping $f: X\to Y$ is called {\it
quasi-compact} if for any open quasi-compact subset $U\subset Y$
its preimage $f^{-1}(U)$ is quasi-compact. Subset  $Z$ is called
{\it retro-compact in } $X$ if the canonical injection
 $Z \hookrightarrow X$ is quasi-compact, and if for
 any open quasi-compact subset  $U \subset X$ the intersection $U
\cap Z$ is quasi-compact.
\end{definition}

Let  $f: X \to S$ be a scheme morphism of finite presentation,
$\MM$ be a quasi-coherent  $\OO_X$-module of finite type.

\begin{definition} \cite[part 1, definition 5.2.1]{RG} $\MM$
is $S$-{\it flat in dimension} $\ge n$ if there exist a
retro-compact open subset $V \subset X$ such that $\dim
(X\setminus V)/S < n$ and if  $\MM|_V$ is $S$-flat module of
finite presentation.
\end{definition}
If $\MM$ is $S$-flat  module of finite presentation and schemes
$X$ and $S$ are of finite type over the field, then any open
subset $V \subset X$ fits to be used in the definition. Setting
$V=X$ we have $X\setminus V=\emptyset$ and $\dim (X\setminus V)/S
=-1-\dim S.$ Consequently,
 $S$-flat  module of finite presentation is flat in dimension  $\ge -\dim S.$

Conversely, let  $\OO_X$-module $\MM$ be  $S$-flat in dimension
$\ge -\dim S$. Then there is an open retro-compact subset
$V\subset X$ such that  $\dim (X\setminus V)/S<-\dim S$ and such
that  $\MM|_V$ is  $S$-flat module. By the former inequality for
dimensions we have  $\dim (X\setminus V)<0$, what implies $X=V$,
and $\MM|_V=\MM$ is $S$-flat.

\begin{definition} \cite[part 1, definition 5.1.3]{RG} Let
 $f: S' \to S$ be a morphism of finite type, $U$ be an open
 subset in  $S$. The morphism $f$ is called
 $U$-{\it admissible blowup} if there exist a closed subscheme
 $Y \subset S$ of finite
 presentation  which is disjoint from $U$
 and such that  $f$ is isomorphic to the blowing up a scheme
$S$ in $Y$.
\end{definition}

\begin{theorem}\label{trg}\cite[theorem 5.2.2]{RG} Let $S$ be a
quasi-compact quasi-separated scheme, $U$ be open quasi-compact
subscheme in  $S$, $f: X\to S$ of finite presentation, $\MM$
$\OO_X$-module of finite type, $n$ an integer. Assume that
$\MM|_{f^{-1}(U)}$ is flat over $U$ in dimension $\ge n$. Then
there exist  $U$-admissible blowup $g: S' \to S$ such that
$g^{\ast} \MM$ is $S'$-flat in dimension $\ge n$.
\end{theorem}

Recall the following
\begin{definition}\cite[definition 6.1.3]{GroDieu} the scheme morphism
$f: X\to Y$ is {\it quasi-separated} if the diagonal morphism
$\Delta_f: X \to X\times_Y X$ is quasi-compact. The scheme  $X$ is
{\it quasi-separated} if it is quasi-separated over $\Spec \Z.$
\end{definition}
If the scheme  $X$ is Noetherian, then any morphism
 $f: X\to Y$ is quasi-compact. Since we work in the category
 of Noetherian schemes, all morphisms of our interest and all arising schemes
 are quasi-compact.

Due to theorem \ref{trg}, there exist a $T_0$-admissible blowing
up $g: \widetilde T \to T$ such that inverse images of sheaves
$\OO_{\widehat \Sigma}$ and $\widehat \E$ are $\widetilde T$-flat.
Namely, in the notation fixed by the following fibred square
\begin{equation}\label{ressq}\xymatrix{\widetilde \Sigma \ar[d]^{\pi} \ar[r]^{\widetilde g}&\widehat \Sigma \ar[d]^f\\
\widetilde T \ar[r]^g & T}
\end{equation}
$\OO_{\widetilde \Sigma}=\widetilde g^{\ast} \OO_{\widehat
\Sigma}$ and $\widetilde \E:=\widetilde g^{\ast} \widehat \E$ are
flat $\OO_{\widetilde T}$-modules.

\begin{remark} Since  $\widehat \E$ is locally free as
 $\OO_{\widehat
\Sigma}$-module, it is sufficient to achieve that $\OO_{\widetilde
\Sigma}=\widetilde g^{\ast} \OO_{\widehat \Sigma}$ be flat as
$\OO_{\widetilde T}$-module. Then the locally free
$\OO_{\widetilde \Sigma}$-module $\widetilde g^{\ast} \widehat \E$
is also flat over $\widetilde T$.
\end{remark}

The epimorphism
\begin{equation}\label{epi}\widetilde g^{\ast}
\sigma\!\!\!\sigma^{\ast} \E \twoheadrightarrow \widetilde
\E\end{equation} induced by the right vertical arrow in
(\ref{depi}), provides quasi-ideality on closed fibres of the
morphism  $\pi$.


The transformation of families we constructed, has a form $$(T,
\L, \E) \mapsto (\pi: \widetilde \Sigma \to \widetilde T,
\widetilde \L, \widetilde \E)$$ and is defined by the commutative
diagram
\begin{equation*}\xymatrix{ T \ar[d]_{g^o}\ar@{|->}[r]& \{(T, \L,
\E)\} \ar[d]\\
\widetilde T \ar@{|->}[r]& \{(\pi: \widetilde \Sigma \to
\widetilde T, \widetilde \L, \widetilde \E)\} }
\end{equation*}
where left vertical arrow is a morphism in the category
$(Schemes_k)^o$. This morphism is dual to the blowup morphism
 $g: \widetilde T \to T$. The right vertical arrow is the map of sets.
Their elements are families of objects to be parametrized. The map
is determined by the procedure of resolution as it developed in
this section.
\begin{remark} The morphism  $g$ is defined by the structure of
the concrete family of coherent sheaves under resolution but not
by the class of families which is an image of the scheme $T$ under
the functor ${\mathfrak f}^{GM}$. Then the transformation as it is
constructed now does not define a morphism of functors.
\end{remark}
\begin{remark} The resolution we constructed for families with
nonreduced base is not applicable to families without locally free
sheaves, because in the suggested procedure the result from
\cite{RG} is involved (theorem \ref{trg}). This result operates
with the notion of flatness in dimension $\ge n$. In particular,
the construction is not applicable to families of nonlocally free
sheaves with zero-dimensional base, as well as to investigate such
components of Gieseker -- Maruyama moduli scheme which do not
contain locally free sheaves.
\end{remark}

\section{Construction of  morphism of  functors}

To obtain the natural transformation of interest it is necessary
to show that flatness of the family $(\pi: \widetilde \Sigma \to
\widetilde T, \widetilde \L, \widetilde \E)$ over $\widetilde T$
implies that the family  $(pr_1 \circ \sigma\!\!\!\sigma: \widehat
\Sigma \to T, \widehat \L, \widehat \E)$ is also flat over $T$.

To solve the descent problem for the property of flatness of a
constructed family $(\pi: \widetilde \Sigma \to \widetilde T,
\widetilde \L, \widetilde \E)$ along the morphism  $g$ we need the
following
\begin{definition} The scheme morphism $h:X \to Y$
{\it has infinitesimal sections} if for any closed point $y\in Y$
and for any zero-dimensional subscheme $Z_Y \subset Y$ supported
at $y$ there is a zero-dimensional subscheme $Z_X \subset X$ such
that the induced morphism $h|_{Z_X}: Z_X \to Z_Y$ is an
isomorphism.
\end{definition}

It is clear that the morphism having infinitesimal sections is
surjective. Any blowup morphism has infinitesimal sections.
Indeed, for any ring  $A$, ideals $I \subset A \supset I_Y$ where
$I_Y$ is an ideal of a zero-dimensional subscheme in $\Spec A$,
and for a canonical morphism onto zeroth component $h^{\sharp}: A
\to \bigoplus_{s\ge 0} I^s$, the following diagram of $A$-modules
commutes
\begin{equation*} \xymatrix{A \ar@{->>}[d] \ar[r]^{h^{\sharp}}& \bigoplus_{s\ge 0} I^s \ar@{->>}[d]\\
A/I_Y \ar@{=}[r]& A/I_Y}
\end{equation*}
The right vertical arrow in this diagram is a projection on the
quotient module of zeroth component.

Now we need the analogue of the well known criterion of flatness
involving Hilbert polynomial, for the case of nonreduced base
scheme.

If  $t\in T$ is closed point of the scheme $T$ corresponding to
the sheaf of maximal ideals ${\mathfrak m}_t\subset \OO_X$, then
we denote $m$th infinitesimal neighborhood of the point $t\in T$
by the symbol $t^{(m)}$.  It is a subscheme defined in $T$ by the
sheaf of ideals ${\mathfrak m}_t^{m+1}$.

\begin{proposition}\label{critF} \cite[theorem 3]{TimAA} Let a
projective morphism of Noetherian schemes of finite type $f: X\to
T$ is include in the commutative diagram
\begin{equation*}\label{triangle}
\xymatrix{X \ar@{^(->}[r]^i \ar[rd]_f& \P^N_T \ar[d]\\
&T}
\end{equation*}
where $i$ is closed immersion. The coherent sheaf of
$\OO_X$-modules $\FF$ is flat with respect to $f$ (i.e. flat as a
sheaf of $\OO_T$-modules) if and only if for an invertible
$\OO_X$-sheaf $\LL$ which is very ample relatively to $T$ and such
that $\LL= i^{\ast} \OO(1),$ for any closed point  $t\in T$ the
function
$$
\varpi_t^{(m)}(\FF, n)=\frac{\chi (\FF \otimes
\LL^n|_{f^{-1}(t^{(m)})})}{\chi(\OO_{t^{(m)}})}
$$
does not depend of the choice of $t\in T$ and of $m\in \N$.
\end{proposition}
\begin{remark} For $m=0$ we have
$\varpi_t^{(m)}(\FF, n)=\chi(\FF \otimes \LL^n|_{f^{-1}(t)})$.
This is Hilbert polynomial of the restriction of the sheaf $\FF$
to the fibre at the point $t$.
\end{remark}
\begin{remark}\label{allst} The proof done in  \cite[theorem
3]{TimAA} implies that all possible zero-dimensional subschemes
supported at the closed point $t$ can be considered instead of
infinitesimal neighborhoods of this point.
\end{remark}

\begin{proposition} Let we are given fibred diagram (\ref{ressq}) of
Noetherian schemes where  $f$ is projective morphism and the
morphism $g$ has infinitesimal sections. Let also the scheme
$\widehat \Sigma$ carry an invertible sheaf $\widehat \L$ which is
very ample relatively to $T$, and its inverse image $\widetilde
\L= \widetilde g ^{\ast} \widehat \L$ is flat relatively to the
morphism  $\pi$. Then the morphism $f$ is also flat and the sheaf
$\widehat \L$ is flat relatively to $f$.
\end{proposition}
\begin{proof}
To verify flatness we use the criterion of the proposition
\ref{critF}. Choose a closed point $t$ in $T$ and a
zero-dimensional subscheme $Z_t$, $\Supp Z_t=t$.

Now choose a zero-dimensional subscheme $Z'_t \subset \widetilde
T$ in the preimage  $g^{-1}Z_t$ such that this zero-dimensional
subscheme maps isomorphically to $Z_t$ under the morphism  $g$.
Consider a following commutative diagram
\begin{equation*}\xymatrix{
&\widetilde \Sigma \ar[dd]^>>>>>>>>\pi \ar[rr]^{\widetilde g}&&
\widehat
\Sigma \ar[dd]^f\\
\pi^{-1}Z'_t \ar[dd]_{\pi_t} \ar@{^(->}[ur]^{\widetilde i}
\ar[rr]^>>>>>>{\widetilde g_t}&& f^{-1}Z_t \ar@{^(->}[ur]^{\widehat i} \ar[dd]\\
&\widetilde T \ar[rr]^<<<<<<<<g&& T\\
Z'_t \ar@{^(->}[ur]^{i'} \ar[rr]_{\thicksim}^{g_t} &&Z_t
\ar@{^(->}[ur]^i}
\end{equation*}
where all skew arrows are closed immersions, left and right
parallelograms and  the rectangle containing $T$ are fibred. Usual
verifying of universality shows that the rectangle with $Z_t$ is
also fibred. This implies that
 $\widetilde g_t:=\widetilde
g|_{\pi^{-1}Z'_t}$ is an isomorphism. Then
$$
h^0(f^{-1}Z_t, \widehat \L^n|_{f^{-1}Z_t})= h^0(f^{-1}Z_t,
\widehat i^{\ast}\widehat \L^n)= h^0(\pi^{-1}Z'_t, \widetilde
g_t^{\ast}\widehat i^{\ast}\widehat \L^n) $$
$$=h^0(\pi^{-1}Z'_t,
\widetilde i^{\ast} \widetilde g^{\ast} \widehat \L^n)=
h^0(\pi^{-1}Z'_t, \widetilde i^{\ast}  \widetilde \L^n)=
h^0(\pi^{-1}Z'_t,  \widetilde \L^n|_{\pi^{-1}Z'_t}).
$$
In particular if  $Z_t=t$ and $Z'_t=\widetilde t$, $g(\widetilde
t)=t$ are reduced points, then Hilbert polynomials of fibres $\chi
(\widehat \L^n|_{f^{-1}(t)})$ and $\chi(\widetilde
\L^n|_{\pi^{-1}(\widetilde t)})$ coincide.

Then for $m\gg 0$ we have $$\chi(\widehat
\L^m|_{f^{-1}Z_t})=h^0(f^{-1}Z_t, \widehat
\L^m|_{f^{-1}Z_t})=h^0(\pi^{-1}Z'_t, \widetilde
\L^m|_{\pi^{-1}Z'_t})=\chi(\widetilde \L^m|_{\pi^{-1}Z'_t}).$$ By
the proposition \ref{critF} in view of the remark  \ref{allst},
since $\pi$ is flat morphism and $\widetilde \L$ provides equal
Hilbert polynomials on its fibres, then $$\chi(\widetilde
\L^m|_{\pi^{-1}Z'_t})=\chi ( \widetilde \L^m|_{\pi^{-1}(\widetilde
t)}) \length (Z'_t)=\chi ( \widehat \L^m|_{f^{-1}(t)})\length
(Z_t)=\chi(\widehat \L^m|_{f^{-1}Z_t}).$$ Hence the morphism  $f$
is also flat what completes the proof of the proposition.
\end{proof}

Then $(T, f: \widehat \Sigma \to T, \widehat \L, \widehat \E)$ is
the required family of semistable admissible pairs with base  $T$.
The performed construction defines the natural transformation of
the functor of semistable torsion-free coherent sheaves to the
functor of admissible semistable pairs and hence it completes the
proof of the theorem 1.

\section{Morphism of moduli schemes}

The developed procedure of standard resolution for a family of
semistable torsion-free coherent sheaves with possibly nonreduced
base, allows to construct the morphism of the Gieseker -- Maruyama
moduli scheme to the moduli scheme of admissible semistable pairs
without categorical considerations.

According to the classical construction of Gieseker -- Maruyama
moduli scheme $\overline M$, choose an integer $m \gg 0$ such that
for each semistable coherent sheaf $E$ the morphism $H^0(S, E
\otimes L^m) \otimes L^{-m} \to E$ is surjective, $H^0(S, E\otimes
L^m)$ is a  $k$-vector space of dimension  $rp(m)$ and  $H^i(S,
E\otimes L^m)=0$ for all $i>0$.

Then take a $k$-vector space  $V$ of dimension $rp(m)$ and form a
Grothendieck scheme  $\Quot ^{rp(n)} (V\otimes L^{-m})$ of
coherent $\OO_S$-quotient sheaves of the form $V \otimes L^{-m}
\twoheadrightarrow E$. Consider its quasi-projective subscheme  $Q
\subset \Quot^{rp(n)} (V\otimes L^{-m})$, corresponding to those
 $E$ which are Gieseker-semistable and torsion-free.

The scheme  $\Quot^{rp(n)}(V \otimes L^{-m})$ is acted upon by the
algebraic group  $PGL(V)$ by linear transformations of the vector
space  $V$. The scheme $\overline M$ is obtained as a (good)
GIT-quotient of the subscheme  $Q$ (which is immersed
$PGL(V)$-equivariantly in $\Quot^{rp(n)}(V \otimes L^{-m})$), by
the action of $PGL(V):$ $\overline M=Q//PGL(V).$

By \cite{Tim6}, the scheme $\widetilde M$ is built up in analogous
way. For an admissible semistable pair $((\widetilde S, \widetilde
L), \widetilde E)$ for $m\gg 0$ there is a closed  immersion  $j:
\widetilde S \hookrightarrow G(V,r)$, determined by the
epimorphism of locally free sheaves $H^0(\widetilde S, \widetilde
E \otimes \widetilde L^m) \boxtimes \widetilde L^{-m} \to
\widetilde E$. Here $G(V,r)$ is Grassmann variety of
$r$-dimensional quotient spaces of the vector space $V$.

Let  $\OO_{G(V,r)}(1)$ be a positive generator of the group $\Pic
G(V,r)$, then \linebreak $P(n):= \chi (j^{\ast} \OO_{G(V,r)}(n))$
is a Hilbert polynomial of the closed subscheme  $j(\widetilde S)
$. Fix the polynomial $P(n)$ and consider the Hilbert scheme
$\Hilb^{P(n)}G(V,r)$ of subschemes in $G(V,r)$ with Hilbert
polynomial equal to $P(n)$. Let $H_0$ be quasi-projective
subscheme in $\Hilb^{P(n)}G(V,r)$ whose points correspond to
admissible semistable pairs. The Grassmann variety $G(V,r)$ and in
induced way the Hilbert scheme $\Hilb^{P(n)}G(V,r)$ are acted upon
by the group $PGL(V)$. This action, as in the case with Gieseker
-- Maruyama compactification, is inspired by linear
transformations of the vector space $V$. The scheme $\widetilde M$
is obtained as a (good)  GIT-quotient of the scheme $H_0$ by the
action of the group $PGL(V)$: $\widetilde M=H_0 // PGL(V)$.

To construct a morphism  $\overline M \to \widetilde M$ we recall
that the Grothendieck' scheme $\Quot:=\Quot^{rp(n)}(V \otimes
L^{-m})$ carries a universal family of quotient sheaves
$$
V \boxtimes \OO_{\Quot \times S} \twoheadrightarrow \E_{\Quot}.
$$
Restrict it to the subscheme $Q\times S \subset \Quot\times S$:
$$
V \boxtimes \OO_{Q\times S} \twoheadrightarrow \E_Q.
$$
The sheaf of $\OO_{Q\times S}$-modules $\E_Q:=
\E_{\Quot}|_{Q\times S}$ provides a family of coherent semistable
torsion-free $\OO_S$-sheaves with base scheme $Q$.

Applying the procedure of standard resolution as developed in this
article, to $(\Sigma=Q\times S, \L=\OO_Q \boxtimes L,\E=\E_Q)$, we
come to the collection of data $((\widehat \Sigma_Q, \widehat
\L_Q), \widehat \E_Q)$. By the universal property of the Hilbert
scheme $\Hilb^{P(n)}G(V,r)$, the family $((\widehat \Sigma_Q,
\widehat \L_Q), \widehat \E_Q)$ induces a morphism $\widehat
\Sigma_Q \to \Univ^{P(n)}G(V,r)$ into the universal subscheme
$$\Univ^{P(n)}G(V,r)\subset \Hilb^{P(n)}G(v,r)\times S,$$ and a
morphism of the base scheme $\mu: Q\to \Hilb^{P(n)}G(V,r)$. These
morphisms are included into the following commutative diagram with
fibred square
\begin{equation*}\xymatrix{\widehat \Sigma_Q \ar[d]_{f} \ar[r]&
\Univ^{P(n)}G(V,r) \ar[d] \ar@{^(->}[r]& \ar[ld]^{pr_1}\Hilb^{P(n)}G(V,r)\times G(V,r)\\
Q \ar[r]^<<<<<{\mu}& \Hilb^{P(n)}G(V,r)}
\end{equation*}
Hence the morphism  $\mu$ decomposes as
$$Q \twoheadrightarrow \mu(Q) \hookrightarrow \Hilb^{P(n)} G(V,r).$$

Now insure that this composite factors through the subscheme
 $H_0\subset \Hilb^{P(n)} G(V,r)$ of admissible
semistable pairs.

Firstly, existence and structure of a morphism $\sigma \!\!\!
\sigma: \widehat \Sigma_Q \to Q \times S$ as a blowing up morphism
in the sheaf of Fitting ideals (\ref{fitt}) guarantees that the
family
 $\widehat \Sigma_Q \to Q$ is formed by admissible schemes.

Secondly, resolution of the sheaf $\E_Q$ into the sheaf $\widehat
\E_Q$ provides Gieseker's semi\-stability of locally free sheaves
in the family $\widetilde \E_Q$. Indeed, in \cite{Tim4} it  is
proven that this is true in the case when the base $Q$ of the
family is considered with reduced scheme structure: $Q=Q_{\red}$.
Since Gieseker's semistability of a coherent sheaf is open
condition in flat families, then if the image of $Q_{\red}$ under
the morphism $\mu: Q \to \Hilb^{P(n)}G(V,r)$ belongs to the open
subset of semistable sheaves, the same is true for the whole of
the scheme $Q$.

Thirdly, the epimorphism  (\ref{epi}) provides quasi-ideality of
sheaves in the family  $\widehat \E_Q$. Applying resolution from
section 2 to the case $T=Q$, we come to quasi-ideality on the
image of the scheme $Q$ under the resolution.

The procedure of resolution constructed in Section 2 and applied
to the quasi-projective scheme  $Q$, leads to the family of
admissible semistable pairs $((f: \widehat \Sigma \to Q, \widehat
\L:={\sigma\!\!\! \sigma}^{\ast}\L \otimes {\sigma\!\!\!
\sigma}^{-1}I \cdot \OO_{Q \times S}), \widehat \E))$. This family
fixes a subscheme $\mu(\widetilde Q)\subset H_0$ in
$\Hilb^{P(n)}G(V,r)$.  It defines
 $PGL(V)$-equivariant composite $Q
\twoheadrightarrow \mu(Q) \subset H_0$. This composite leads to
the morphism of GIT-quotients $\kappa: \overline M \to \widetilde
M$.


\bigskip
\begin{thebibliography}{1}


\bibitem{Tim1}
N.\,V.~Timofeeva, {\it On a new compactification of the moduli of
vector bundles on a surface},  Sb. Math., {\bf 199:7}(2008),
1051--1070.

\bibitem{Tim2}   N.\,V.~Timofeeva, {\it On a new compactification of
the moduli of vector bundles on a surface. II},  Sb. Math. {\bf
200:3}(2009),  405--427.

\bibitem{Tim3}
N.\,V.~Timofeeva, {\it On degeneration of surface in Fitting
compactification of moduli of stable vector bundles}, Math. Notes,
{\bf 90}(2011), 142--148.




\bibitem{Tim4}   N.\,V.~Timofeeva, {\it On a new
compactification of the moduli of vector bundles on a surface.
III: Functorial approach},  Sb. Math., {\bf 202:3}(2011), 413 --
465.

\bibitem{Tim6}   N.\,V.~Timofeeva, {\it On a new
compactification of the moduli of vector bundles on a surface. IV:
Nonreduced moduli},  Sb. Math., {\bf 204:1}(2013), 133--153.


\bibitem{Tim7} N.\,V.~Timofeeva, {\it On a new
compactification of the moduli of vector bundles on a surface. V:
Existence of universal family},  Sb. Math., {\bf 204:3}(2013),
411--437.


\bibitem{Gies}  D.~Gieseker, {\it On the moduli of vector
bundles on an algebraic surface},  Annals of Math., {\bf
106}(1977), 45--60.

\bibitem{Mar}  M.~Maruyama, {\it Moduli of stable sheaves, II},
 J. Math. Kyoto Univ. (JMKYAZ), {\bf 18-3}(1978),
557--614.




\bibitem{HL}  D.~Huybrechts, M.~Lehn, {\it The geometry of
moduli spaces of sheaves}, Aspects Math., E31.  Vieweg,
Braunschweig, 1997.

\bibitem{O'Gr} K.\,G.~O'Grady, {\it Moduli of vector bundles
on projective surfaces: some basic results}, Inv. Math., {\bf
123}(1996), 141--207.

\bibitem{O'Gr1} K.\,G.~O'Grady, {\it Moduli of
vector-bundles on surfaces}, preprint ArXiv:alg-geom9609015, 20
Sep 1996.

\bibitem{Gies-Li}   D.~Gieseker, J.~Li, {\it Moduli of high
rank vector bundles over surfaces},  Journal of the Amer. Math.
Soc., {\bf 9:1}(1996),  107--151.

\bibitem{Eisen}  D.~Eisenbud, {\it Commutative algebra. With a
view toward algebraic geometry}, Grad. Texts in Math., 150,
Springer-Verlag, New York -- Berlin, 1995.


\bibitem{Tim} N. V. Timofeeva,  {\it Compactification in Hilbert scheme
of moduli scheme of stable 2-vector bundles on a surface},  Math.
Notes,  {\bf 82:5}(2007), 677 -- 690.

\bibitem{SGAI} {\it A. Grothendieck}, Seminaire de G\'{e}om\'{e}trie
Alg\'{e}brique du Bois Marie, 1960 -- 61. I: Rev\^{e}tements
\'{e}tales et groupe fondamental. Lecture Notes in Math., Springer
-- Verlag, Berlin -- Heidelberg -- New York, 1971.


\bibitem{Tikh} A. S. Tikhomirov, {\it The variety of complete pairs of
zero-dimensional subschemes of an algebraic surface}, Izvestiya:
Mathematics, {\bf 61:6}(1997),  1265 -- 1291.

\bibitem{EGAIII} A. Grothendieck, {\it \'{E}l\'{e}ments de
G\'{e}om\'{e}trie Alg\'{e}brique III: \'{E}tude Cohomologique des
Faiceaux Coh\'{e}rents, Premi\`{e}re Partie }/R\'{e}dig\'{e}s avec
la collaboration de J. Dieudonn\'{e}, Inst. Hautes \'{E}tudes Sci.
Publ. Math\'{e}matiques, No. 11, IH\'{E}S, Paris, 1961.

\bibitem{RG}  M.~Raynaud, L.~Gruson,  {\it Crit\`{e}res de
platitude et de projectivit\'{e}: Techniques de "platification"\,
d'un module}, Inv. Math., {\bf 13}(1971),  1--89.

\bibitem{GroDieu}  A. Grothendieck, J.A. Dieudonn\'{e},
{\it El\'{e}ments de G\'{e}om\'{e}trie Alg\'{e}brique, I},
Springer-Verlag, Berlin -- Heidelberg -- New York, 1971.

\bibitem{TimAA}  N.\,V.~Timofeeva,  {\it Infinitesimal criterion for
flatness of projective morphism of schemes}, S.-Peters\-burg Math.
J. (Algebra i Analiz), {\bf 26-1}(2014), 185--195.



\end{thebibliography}
\end{document}